\documentclass[12pt,a4paper,leqno,verbatim]{amsart}
\newcounter{minutes}
\divide\time by 60
\newcounter{hours}
\multiply\time by 60 \addtocounter{minutes}{-\time}

\usepackage{amssymb}
\usepackage{hyperref}
\usepackage{graphicx}
\usepackage{color}
\date{}
\newfont{\cyrilic}{wncyr10 scaled 1000}
\title[{New means generated by generalized trigonometric functions}]
{On certain new means generated by generalized trigonometric functions}

\author[J. S\'andor]{J\'ozsef S\'andor}
\address{Babe\c{s}-Bolyai University,
Department of Mathematics,
Str. Kogalniceanu nr. 1
400084 Cluj-Napoca, Romania}
\email{jsandor@math.ubbcluj.ro}

\author[B.A Bhayo]{Barkat Ali Bhayo}
\address{Department of Mathematics, Sukkur IBA University, Airport road Sukkur, Sindh, Pakistan\newline
Faculty of Natural Sciences, Sabanci University, 34956 Tuzla/Istanbul, Turkey}
\email{barkat.bhayo@iba-suk.edu.pk}

\newcommand{\comment}[1]{}

\swapnumbers
\theoremstyle{plain}

\newtheorem{theorem}[equation]{Theorem}
\newtheorem{lemma}[equation]{Lemma}

\newtheorem{corollary}[equation]{Corollary}

\newtheorem{remark}[equation]{Remark}

\numberwithin{equation}{section}

\begin{document}
\font\fFt=eusm10 
\font\fFp=eusm5  
\def\K{\mathchoice
{\hbox{\,\fFt K}}
{\hbox{\,\fFt K}}
{\hbox{\,\fFa K}}
{\hbox{\,\fFp K}}}
\def\T{\mathchoice
{\hbox{\,\fFt T}}
{\hbox{\,\fFt T}}
{\hbox{\,\fFa T}}
{\hbox{\,\fFp T}}}

\def\thefootnote{}
\footnotetext{ \texttt{\tiny File:~\jobname .tex,
          printed: \number\year-\number\month-\number\day,
          \thehours.\ifnum\theminutes<10{0}\fi\theminutes}
} \makeatletter\def\thefootnote{\@arabic\c@footnote}\makeatother

\allowdisplaybreaks

\begin{abstract}
In this paper, authors generalize logarithmic mean $L$, Neuman-S\'andor $M$, two Seiffert means $P$ and $T$ 
as an application of generalized trigonometric and hyperbolic functions. Moreover, several two-sided inequalities 
involving these generalized means are established.
\end{abstract}

\maketitle




\section{Introduction}
For the definition of new means involved in our formulation we introduce some special functions
and notation. The \emph{Gaussian hypergeometric function} is defined by
$$F\left(a,b;c;z\right)={}_2F_1\left(a,b;c;z\right)=\sum^\infty_{n=0}\frac{(a,n)(b,n)}
{(c,n)}\frac{z^n}{n!},\quad |z|<1,$$
where $(a,n)$ denotes the shifted factorial function
$$(a,n)=a(a+1)(a+2)\ldots (a+n-1),\quad n=1,2,3,\ldots,$$
and $(a,0)=1$ for $a\neq 1$. For the applications of this function in 
various fields of the mathematical
and natural sciences, reader is referred to \cite{askey}.

Special functions, such as classical \emph{gamma function} $\Gamma$, the \emph{digamma function} $\psi$ and
the \emph{beta function} $B(.\,,.)$ have close relation with hypergeometric function. For $x, y > 0$, these functions are
defined by
$$\Gamma(x)=\int^\infty_0 e^{-t}t^{x-1}\,dt,\quad   \psi(x)=\frac{\Gamma^{'}(x)}{\Gamma(x)},
\quad B(x,y)=\frac{\Gamma(x)\Gamma(y)}{\Gamma(x+y)},$$
respectively.
The hypergeometric function can be represented in the integral form as follows
\begin{equation}\label{integral}
F(a,b;c;z)=\frac{\Gamma(c)}{\Gamma(b)(c-b)}\int_0^1{t^{b-1}(1-t)^{c-b-1}(1-zt)^{-a}dt}.
\end{equation}

The eigenfunction $\sin_p$ of the so-called one-dimensional $p$-Laplacian problem
\cite{dm}
$$-\Delta_p u=-\left(|u'|^{p-2}u'\right)'
=\lambda|u|^{p-2}u,\,u(0)=u(1)=0,\ \ \ p>1,$$
is the inverse function of $F_p:[0,1]\to \left[0,\frac{\pi_p}{2}\right]$, defined as
$$F_p(x)={\rm arcsin}_p(x)=\int^x_0(1-t^p)^{-\frac{1}{p}}dt,$$
where
$$\pi_p=2{\rm arcsin}_p(1)=\frac{2}{p}\int^1_0(1-s)^{-\frac{1}{p}}s^{\frac{1}{p}-1}ds=\frac{2}
{p}\,B\left(1-\frac{1}{p},\frac{1}{p}\right)=\frac{2 \pi}{p\,\sin\left(\frac{\pi}{p}\right)}\,.$$
The function ${\rm arcsin}_p$ is called the generalized inverse sine function, and its inverse function $\sin_p:[0,\pi_p/2]\to[0,1]$ is called generalized sine function. For $x\in[\pi_p/2,\pi_p]$, one can extends the function $\sin_p$ to $[0,\pi_p]$ by defining
$\sin_p(x)=\sin(\pi_p-x)$,
and further extension cab be achieved on $\mathbb{R}$ by oddness and $2\pi$-periodicity. 
The range of $p$ is restricted to 
$(1, \infty)$ because only in this case $\sin_p(x)$ can be made periodic like usual sine function.

Similarly, the other generalized inverse trigonometric and hyperbolic functions
${\rm arccos}_p:(-1,1)\to \left(-a_p,a_p\right),\,{\rm arctan}_p:(-\infty,\infty)\to (-a_p,a_p),\,{\rm arcsinh}_p:(-\infty,\infty)\to(-\infty,\infty),\,
{\rm arctanh}_p:(-1,1)\to (-\infty,\infty)$
are defined as follows
\begin{equation}\label{pintegrals}
\begin{aligned}
{\rm arccos}_p(x)&=&\int^{(1-x^p)^{\frac{1}{p}}}_0(1-|t|^p)^{-\frac{1}{p}}dt,\quad 
{\rm arctan}_p(x)=\int^x_0(1+|t|^p)^{-1}dt,\\
{\rm arcsinh}_p(x)&=&\int^x_0(1+|t|^p)^{-\frac{1}{p}}dt,\quad
{\rm arctanh}_p(x)=\int^x_0(1-|t|^p)^{-1}dt,
\end{aligned}
\end{equation}
where $a_p=\pi_p/2$.
Above inverse generalized trigonometric and hyperbolic functions coincide with usual trigonometric and hyperbolic functions for $p=2$.

\section{\bf Generalization of means and main result}

For two positive real numbers $a$ and $b$, we define arithmetic mean $A$, geometric mean $G$, logarithmic mean $L$, 
two Seiffert means $P$ and $T$, and Neuman-S\'andor mean $M$ introduced in \cite{ns1} as follows,
$$A=A(a,b)=\frac{a+b}{2},\quad G=G(a,b)=\sqrt{ab},$$
$$L=L(a,b)=\frac{a-b}{\log(a)-\log(b)},\quad a\neq b,$$
$$P=P(a,b)=\frac{a-b}{2{\rm arcsin}\left(\frac{a-b}{a+b}\right)},$$
$$T=T(a,b)=\frac{a-b}{2{\rm arctan}\left(\frac{a-b}{a+b}\right)},$$
$$M=M(a,b)=\frac{a-b}{2{\rm arcsinh}\left(\frac{a-b}{a+b}\right)}.$$

The arithmetic-geometric mean $AG(a, b)$ of two real numbers $a$ and $b$ is defined as follows: Let
us consider the sequences $\{a_n\}$ and $\{b_n\}$ satisfying
$$a_{n+1}=\frac{a_n+b_n}{2},\quad b_{n+1}=\sqrt{a_nb_n},\quad n=0,1,2,\ldots$$
with $a_0=a$ and $b_0=b$.

In \cite{bhatia}, Bhatia and Li generalized the logarithmic mean $L$ and arithmetic-geometric mean $AG(a, b)$ 
by introducing an interpolating family of means $\mathbb{M}_p(a, b)$, defined by
$$\frac{1}{\mathbb{M}_p(a, b)}=n_p\int_0^\infty{\frac{dt}{((t^p+a^p)(t^p+b^p))^{1/p}}},\quad p\in(0,\infty),$$
where $n_p=\int_0^\infty{\frac{dt}{(1+t^p)^{2/p}}}.$ Moreover, 
$$\mathbb{M}_0(a, b)=\lim_{p\to 0}\mathbb{M}_{p}(a, b)=\sqrt{ab},$$
$$\mathbb{M}_1(a, b)=L(a,b) \quad {\rm and}\quad \mathbb{M}_2(a, b)=AG(a,b).$$

In \cite{nueman1, nueman2}, Neuman generalized the logarithmic mean $L$, two Seiffert means $P$ and $T$,
and the Neuman-S\'andor mean $M$ by introducing the the $p$-version of the Schwab-Borchardt mean $SB_p$ as follows
$$L_p=L_p(a,b)= SB_p(A_{p/2},G)=\frac{A_{p/2}v_p}{{\rm arctanh}_p(v_p)},$$
$$P_p=P_p(a,b)= SB_p(G,A_{p/2})=\frac{A_{p/2}v_p}{{\rm arcsin}_p(v_p)},$$
$$T_p=T_p(a,b)= SB_p(A_{p/2},A_p)=\frac{A_{p/2}v_p}{{\rm arctan}_p(v_p)},$$
$$M_p=M_p(a,b)= SB_p(A_p,A_{p/2})=\frac{A_{p/2}v_p}{{\rm arcsinh}_p(v_p)},$$
where
$$SB_p(a,b)=b\, F\left(\frac{1}{p},\frac{1}{p};1+\frac{1}{p},1-\left(\frac{a}{b}\right)^p\right)^{-1},$$
$$v_p=\frac{|x^{p/2}-y^{p/2}|}{x^{p/2}+y^{p/2}},$$
and $A_p=A_p(a,b)$ is a power mean of order $p$.

Motivated by the work of Neuman \cite{nueman1, nueman2}, Bhatia and Li \cite{bhatia}, here we give a natural and new generalization 
of $L,\, P,\,T$ and $M$ by utilizing the generalized trigonometric and generalized hyperbolic functions as follows.  

\bigskip

\noindent{\bf Generalization of means.}
For  $p\geq 2$ and $a>b>0$, we define
\begin{equation}\label{pmeans}
  \begin{aligned}
\tilde{P}_p=\tilde{P}_p(a,b)&=&\frac{a-b}{2{\rm arcsin}_p\left(\frac{a-b}{a+b}\right)}=\frac{x}{{\rm arcsin}_p(x)}A,\\
\tilde{T}_p=\tilde{T}_p(a,b)&=&\frac{a-b}{2{\rm arctan}_p\left(\frac{a-b}{a+b}\right)}=\frac{x}{{\rm arctan}_p(x)}A,\\
\tilde{L}_p=\tilde{L}_p(a,b)&=&\frac{a-b}{2{\rm artanh}_p\left(\frac{a-b}{a+b}\right)}=\frac{x}{{\rm artanh}_p(x)}A,\\
\tilde{M}_p=\tilde{M}_p(a,b)&=&\frac{a-b}{2{\rm arsinh}_p\left(\frac{a-b}{a+b}\right)}=\frac{x}{{\rm arsinh}_p(x)}A,\\
\end{aligned}
\end{equation}
where $x=(a-b)/(a+b)$. 
By utilizing \cite[Lemma 1]{bbp}, the above functions can be expressed in terms of hypergeometric functions as follows,
\begin{eqnarray*}
\tilde{P}_p&=&\frac{A}{F\left(\frac{1}{p},\frac{1}{p};1+\frac{1}{p};x^p\right)},\\
\tilde{T}_p&=&\frac{A \cdot (1+x^p)^{1/p}}{F
\left(\frac{1}{p},\frac{1}{p};1+\frac{1}{p};\frac{x^p}{1+x^p}\right)},\\
\tilde{L}_p&=&\frac{A}{F\left(1,\frac{1}{p};1+\frac{1}{p};x^p\right)},\\
\tilde{M}_p&=&\frac{A \cdot (1+x^p)^{1/p}}
{F\left(1,\frac{1}{p};1+\frac{1}{p};\frac{x^p}{1+x^p}\right)}.
\end{eqnarray*}
where $x=(a-b)/(a+b)$ and $a>b>0$.
Now we are in the position to state our main result. Our main result reads as follows.
\begin{theorem}\label{thm1-2704} 
For $p\geq 2$ and $a>b>0$, the functions
$\tilde{P}_p,\, \tilde{T}_p,\, \tilde{L}_p$ and $\tilde{M}_p$
define a mean of two variables $a$ and $b$.
\end{theorem}

\begin{theorem}\label{monoton-thm}
For $x \in (0,1)$ and $1<p<q$, 
\begin{enumerate}
\item the function  $f_1(x)= \arcsin _p(x)/\arcsin _q(x)$   is strictly increasing,
\item the function  $f_2(x)= {\rm arcsinh}_p(x)/{\rm arcsinh}_q(x)$ is strictly decreasing,
\item the function  $f_3(x)= {\rm arctanh} _p(x)/{\rm arctanh} _q(x)$ is strictly increasing,
\item the function $f_4(x)= \arctan _p(x)/\arctan_q(x)$ is strictly decreasing (increasing) for $x \in (x,x_0)$ $(x\in(x_0,1))$,  
where $x_0$ is the unique solution  in $(0,1)$ to the equation $qx^{q-p} +(q-p)x^q-p=0.$
\end{enumerate}

In particular, for $2\leq p<q$ one has 
\begin{enumerate}
\item $\displaystyle\frac{\pi_q}{\pi_p} <\frac{\tilde{ P}_p}{\tilde{ P}_q}<1$,
\item $\displaystyle 1< \frac{\tilde{ M}_p}{\tilde{ M}_q}< \frac{c_q}{c_p}$,
\item $\displaystyle 1< \frac{\tilde{ L}_p}{\tilde{ L}_q}< \frac{q}{p}$,
\item $\displaystyle\frac{\tilde{T}_p}{\tilde{T}_q}<(>) \frac{ b_q}{b_p}$  for  $x \in (0,x_0) (x \in (x_0,1))$,
where $x= (a-b)/(a+b), (a>b>0)$,    $a,  b$ are the arguments of means, i.e.   $\tilde{T}_r=\tilde{T}_r(a,b)$,
and
\begin{equation}
b_p={\rm arctan}_p(1)=\frac{1}{2p}\left(\psi\left(\frac{1+p}{2p}\right)
-\psi\left(\frac{1}{2p}\right)\right)=2^{-\frac{1}{p}}
F\left(\frac{1}{p},\frac{1}{p};1+\frac{1}{p};\frac{1}{2}\right),
\end{equation}
 $$c_p={\rm arcsinh}_p(1)=\left(\frac{1}{2}\right)^{\frac{1}{p}}F\left(1,\frac{1}{p};1+\frac{1}{p},\frac{1}{2}\right).$$
\end{enumerate}
\end{theorem}

\begin{theorem}\label{PpMp1} For $a>b>0$ and $x=(a-b)/(a+b)$, the following inequalities hold true,
\begin{equation}\label{(15-16)}
\tilde{P}_p  \tilde{M}_p \leq (\tilde{P}_{2p})^2\leq k(x,p) \tilde{P}_p \tilde{M}_p,
\end{equation}
where $$k(x,p)= \frac{((1+x^p)^{2/p}+ (1-x^p)^{2/p})^2}{ 4(1-x^{2p})^{1/p}}. $$
\end{theorem}

\begin{theorem}\label{PpMp2} For $a>b>0$ and $x=(a-b)/(a+b)$, the following inequalities
\begin{equation}\label{(23)}
\frac{1}{\tilde{P}_p} +\frac{r}{\tilde{M}_p} \leq \frac{r+1}{\tilde{P}_{2p}},
\end{equation}
 and
\begin{equation}\label{(26)}
(\tilde{P}_{2p})^{2p} \left(\frac{1}{(\tilde{P}_p)^p} + \frac{1}{(\tilde{M}_p)^p}\right) \leq  R(x,p),
\end{equation}
hold true,
where $$r=r(p,x)= \frac{(1+x^p)}{(1-x^p)}^{1/(2p)},$$
$$R=R(x,p)=[(1-x^{2p}) ^{1/(2p)}+ (1-x^{2p})^{-1/(2p)}]^{1/(2p)}/ 2^{2p-1}.$$
\end{theorem}

\begin{theorem}\label{TpLp1} For $a>b>0$ and $x=(a-b)/(a+b)$, one has,
\begin{equation}\label{(20)}
A  \tilde{L}_{2p}\geq \tilde{T}_p \tilde{L}_p,
\end{equation}
\begin{equation}\label{(20b)}
\frac{A^2}{\tilde{T}_p \tilde{L}_p} -\frac{A}{\tilde{L}_{2p}} \leq  \frac{x^{2p}}{4(1-x^{2p})}.
\end{equation}
\end{theorem}

\begin{theorem}\label{TpLp-single1} For $a>b>0$ and $x=(a-b)/(a+b)$, one has,
\begin{equation}\label{(31)}
 \frac{4(1+x^p)(1+ x^p/(p+1))}{(x^p+2)^2}\leq  \frac{\tilde{T}_p}{A}  \leq  1+\frac{x^p}{1+p},
\end{equation}
and
\begin{equation}\label{(32)}
\frac{4(1-x^p)(1- x^p/(1+p))}{(2-x^p)^2}\leq  \frac{\tilde{L}_p}{A}  \leq  1-\frac{x^p}{1+p}.
\end{equation}
\end{theorem}

\begin{theorem}\label{TpLp-single2} For $a>b>0$ and $x=(a-b)/(a+b)$, one has,
\begin{equation}\label{APp}
 \frac{p\,x}{B\left(1/p, 1+1/p\right)} \leq \frac{A}{\tilde{P}_p} \leq  \frac{p\,x \,B(1/p, 1+1/p)(2-x^p)^2]}{4(1-x^p)B(1/p, 1+1/p)]}   
\end{equation}
and
\begin{equation}\label{AMp}
\frac{x}{j(x,p)}\leq  \frac{A}{\tilde{M}_p} \leq \frac{x(1+(1+x^p)^{1/p})}{4j(x,p)(1+x^p)^{1/p}},
\end{equation}
where $j(x,p)=\int_0^x{(1+t^p)^{1/p}}dt$.
\end{theorem}

The paper is organized as follows. In section 1, we give the definition of the special functions involved in our formulation. Section 2 is dedicated to the definition of new means and the statement of the main result. Section 3 consists of preliminary earlier and related results, which will be used in the proving procedures sequel. Section 4 gives the proof of the main result.
\section{\bf Preliminaries and lemmas}

\begin{lemma}\cite[Theorem 2]{avv1}\label{lem0}
For $-\infty<a<b<\infty$,
let $f,g:[a,b]\to \mathbb{R}$
be continuous on $[a,b]$, and be differentiable on
$(a,b)$. Let $g^{'}(x)\neq 0$
on $(a,b)$. If $f^{'}(x)/g^{'}(x)$ is increasing
(decreasing) on $(a,b)$, then so are
$$\frac{f(x)-f(a)}{g(x)-g(a)}\quad and \quad \frac{f(x)-f(b)}{g(x)-g(b)}.$$
If $f^{'}(x)/g^{'}(x)$ is strictly monotone,
then the monotonicity in the conclusion
is also strict.
\end{lemma}

For the proof of the following lemma, see \cite[Theorem 2.1]{bbv}.
\begin{lemma}\label{lem1} For $x\in(0,1)$, 
\begin{enumerate}
\item the function
$p\mapsto {\rm arcsin}_p(x)$ and $p\mapsto {\rm arctanh}_p(x)$
are strictly decreasing and $\log$-convex on $(1,\infty)$. Moreover, 
$p\mapsto {\rm arcsin}_p(x)$ is strictly geometrically convex on $(1,\infty)$.
\item The function $p\mapsto {\rm arctan}_p(x)$ is strictly increasing and concave on $(1,\infty)$.
\end{enumerate}
\end{lemma}
It is easy to observe that the function $p\mapsto {\rm arcsinh}_p(x)$ is strictly decreasing on 
$(1,\infty)$.

\begin{lemma}\label{Lemma1j} We have
\begin{enumerate}
\item The function   $f(t)= (1+t^p)^{-1}$  is strictly decreasing  for $t \in (0,1)$,
\item The function  $g(t)= (1-t^p)^{1}$ is strictly increasing on $(0,1)$,
\item The function $h(t)= (1-t^p)^{-1/p}$ is strictly increasing on $(0,1)$,
\item The function  $s(t)=  (1+t^p)^{-1/p}$ is strictly decreasing on $(0,1)$.
\end{enumerate}
\end{lemma}
\begin{proof} These are immediate consequences of definitions. \end{proof}

For easy reference we recall some well-known inequalities from the literature as follows.

\bigskip

\noindent{\bf Cauchy-Bouniakowski inequality.} If $f,g:[a,b]\to \mathbb{R}$ are integrable, then
\begin{equation}\label{(1)} 
\displaystyle\left(\int_a^b{f(x)g(x)}dx\right)^2 \leq \int_a^b{f(x)^2}dx \int_a^b{g(x)^2}dx.
\end{equation}

\bigskip

\noindent{\bf P\'olya-Szeg\H{o} inequality.} If $f,g:[a,b]\to \mathbb{R}$ are integrable, 
and for all $x\in[a,b]$
$$0<\alpha<f(x)<A,\quad 0<\beta<g(x)<B,$$
then
\begin{equation}\label{(2)}
 \frac{\int_a^b{f(x)^2}dx \int_a^b{g(x)^2}dx}{\left(\int_a^b{f(x)g(x)}dx\right)^2}  \leq   K(\alpha,A,\beta,B),
\end{equation}
where $$K=K(\alpha,A,\beta,B)=  \frac{1}{4}\left(\sqrt{\frac{AB}{\alpha\beta}} + \sqrt{\frac{\alpha\beta}{AB}}\right)^2.$$

\bigskip

\noindent{\bf Chebyshev's inequality.} Let $f,g:[a,b]\to \mathbb{R}$ be integrable. 
If $f$ and $g$ have same type of monotonicity, 
then
\begin{equation}\label{(3)}
\int_a^b{f(x)}dx\cdot \int_a^b{g(x)}dx  \leq   (b-a)\int_a^b{f(x)g(x)}dx.
\end{equation}
If $f$ and $g$ have distinct type of monotonicity, then  
\begin{equation}\label{(4)}
\int_a^b{f(x)}dx\cdot \int_a^b{g(x)}dx  \geq   (b-a)\int_a^b{f(x)g(x)}dx.
\end{equation}

\bigskip

\noindent{\bf Gr\"uss inequality.} If $f,g:[a,b]\to \mathbb{R}$ are integrable,
and for all $x\in[a,b]$
$$0<\alpha<f(x)<A,\quad 0<\beta<g(x)<B,$$
then
\begin{equation}\label{(5)}
\left|(b-a)\int_a^b{f(x)g(x)}dx-\int_a^b{f(x)}dx\cdot \int_a^b{g(x)}dx\right|  \leq   \frac{(b-a)^2}{4}\cdot (A-\alpha)(B-\beta).
\end{equation}

\bigskip

\noindent{\bf Minkowski's inequality.} Let $f,g:[a,b]\to \mathbb{R}$ be integrable and $f,g>0$. Write  

$$h_t(f)=\left(\int_a^b{f(x)^t}dx\right)^{1/t}.$$
Then one has
\begin{equation}\label{(6)}
h_t(f+g)\leq h_t(f)+h_t(g),\quad {\rm for}\quad t\geq 1,
\end{equation}
\begin{equation}\label{(7)}
h_t(f+g)\geq h_t(f)+h_t(g),\quad {\rm for}\quad t\leq 1.
\end{equation}

\bigskip

\noindent{\bf Diaz-Metcalf inequality.} Let $f,g:[a,b]\to \mathbb{R}$ be integrable and suppose that 
there exist constants $m$ and $M$ such that 
$$m\leq g(x)/f(x) \leq M.$$ Then one has
\begin{equation}\label{(8)}
\int_a^b {g^2(x)}dx+m\cdot M\cdot \int_a^b {f^2(x)}dx\leq (m+M)\cdot \int_a^b {f(x)g(x)}dx.
\end{equation}


\section{\bf Proof of main result and corollaries}

\noindent{\bf Proof of Theorem \ref{thm1-2704}.} It is enough to prove that for $p\geq 2$ the following inequalities
\begin{equation}\label{ineq2803-a}
L\leq \tilde{L}_p<\tilde{P}_p<A <\tilde{M}_p< \tilde{T}_p \leq Q
\end{equation}
hold true, where $Q=Q(a,b)=\sqrt{(a^2+b^2)/2}$ is root square mean.
For $p>1$ and $x\in(0,1)$, the following inequalities
$$ {\rm arctan}_p(x)< {\rm arcsinh}_p(x)< {\rm arcsin}_p(x)< {\rm arctanh}_p(x)$$ 
(see \cite[Lemma 9]{bv-eigen}) imply that 
\begin{equation}\label{ineq2803}
\tilde{L}_p< \tilde{P}_p < \tilde{M}_p< \tilde{T}_p.
\end{equation}
It is sufficient to prove that $\tilde{L}_p$ and $\tilde{T}_p$ are means.
Since $\tilde{P}_p < A$,  and from the monotonicity of $x/{\rm arcsinh}_p(x)$ we get  
$M_p >A$, so \eqref{ineq2803} can be completed as:
$$\tilde{L}_p< \tilde{P}_p <A< \tilde{M}_p< \tilde{T}_p ,\quad p>1.$$
Clearly, ${\rm arctanh}_p(x)\leq {\rm arctanh} (x)$. Thus 
$x/{\rm arctanh}_p(x)\geq x/{\rm arctanh}(x)$ for $x \in (0,1)$, implying that
 $\tilde{L_p}/A \geq L/A$, so 
$\tilde{L_p}\geq L$.
Let $x=(a-b)/(a+b)$, then it is easy to see that one has the following identity
$$\frac{Q}{A}=\sqrt{1+x^2}.$$
The last inequality $\tilde{T_p}/A <Q/A= \sqrt{1+x^2}$ in \eqref{ineq2803-a} can be written as 
$x/{\rm arctan}_p(x)< \sqrt{1+x^2}$, or equivalently,
$$\int_0^x{\frac{1}{1+ t^p}}dt> \frac{x}{\sqrt{1+x^2}}.$$
Since  $1+t^p\leq  1+t^2$  (by $ p\geq 2$ and $t \in (0,1)$), we get $1/(1+t^p) \geq  1/(1+t^2)$, so 
${\rm actan}_p(x)\geq {\rm arctan} (x)$.
 Thus  $$\frac{x}{\rm arctan}_p(x) \leq \frac{x}{{\rm arctan}(x)} =\frac{T}{A} <\frac{Q}{A},$$ by the known inequality  $T< Q$.
Therefore, $\tilde{T}_p<Q$,
since we have also $\tilde{T}_p\geq T$, one has $ T\leq \tilde{T}_p< Q$, with equality only for $p=2$. This completes the proof of 
\eqref{ineq2803-a}.
$\hfill\square$

\begin{corollary} For $p\geq 2$, $x,y>0$ with $x\neq y$, we have
$$\frac{x+y}{2 \left(1-\alpha\log \left(1-\left(\frac{x-y}{x+y}\right)^p\right)\right)}
<
\tilde{L}_p(x,y)<
\frac{x+y}{2 \left(1-\beta\log \left(1-\left(\frac{x-y}{x+y}\right)^p\right)\right)}$$
and 
$$\frac{x+y}{2\left(1+\alpha\log \left(1+\left(\frac{x-y}{x+y}\right)^p\right)\right)}u
<\tilde{M}_p(x,y)<
\frac{x+y}{2 \left(1+\beta\log \left(1+\left(\frac{x-y}{x+y}\right)^p\right)\right)}u$$
where 
$$\alpha =1/p,\quad \beta=1/(1+p),\quad u=\left (1 + \left (\frac {x - y} {x + y} \right)^p \right)^{-1/p}$$
\end{corollary}

\begin{proof}
Proof follows easily from \cite[Theorem 2]{bv-eigen}.
\end{proof}

The proof of following three corollaries follow easily from \cite[Theorem 1]{bbp}, Lemma \ref{lem1}, and \cite[Corollary 2.2]{bbv}, respectively.

\begin{corollary} For $p\geq 2$, $x,y>0$ with $x\neq y$, we have
$$\sqrt[p]{1-\left(\frac{x-y}{x+y}\right)^p}\tilde{P}_p(x,y)  <\tilde{L}_p(x,y)  <         
\frac{\tilde{P}_p(x,y)}{A(x,y)^{p-1}} .$$
\end{corollary}

\begin{corollary}\label{meanp-turan} For $p\geq 3$, we have the following Tur\'an type 
inequalities for the means $\tilde{P}_p,\,\tilde{L}_p,\,\tilde{T}_p$,
$$\tilde{P}_p^2>\tilde{P}_{p-1}\tilde{P}_{p+1},$$
$$\tilde{L}_p^2>\tilde{L}_{p-1}\tilde{L}_{p+1},$$
$$\tilde{T}_p^2<\tilde{T}_{p-1}\tilde{T}_{p+1}.$$
\end{corollary}

It also follows from Lemma \ref{lem1} that for $p,q\geq 2$, we have 
$$\tilde{P}_{\sqrt{pq}}\geq \sqrt{\tilde{P}_{p}\tilde{P}_{q}},$$
where equality holds for $p=q$.

\begin{corollary}\label{special-case} One has
$$P<\tilde{P}_3^2/\tilde{P}_{4}<\tilde{P}_3^2/\tilde{L}_{4},$$
$$L<\tilde{L}_3^2/\tilde{L}_{4}<\tilde{L}_3^2/L,$$
$$T>\tilde{T}_3^2/\tilde{T}_{4}>\tilde{T}_3^2/T.$$
\end{corollary}

\bigskip

\begin{corollary}\label{thm2-old} For $p\geq 2$ and $a>b>0$, we have
$$(2/\pi_p)A<\tilde{P}_p<P_{2p}<A.$$
\end{corollary}

\begin{proof} It follows from Lemma \ref{lem0} that the function $t/{\rm arcsin}_p(t)$ is decreasing in
$t\in(0,1)$. By using l'H\^opital rule, we get $\lim_{t\to 0}(t/{\rm arcsin}_p(t))=1$ and
$\lim_{t\to 1}(t/{\rm arcsin}_p(t))=(p\sin(\pi/p)/\pi)$. This implies the first inequality, the second inequality
follows from Lemma \ref{lem1}, and the proof of third inequality follows from first one.
\end{proof}

\begin{theorem}\label{thm3-old} For $p\geq$ and $a>b>0$, we have
\begin{enumerate}
\item $\frac{b_p}{c_p}\tilde{T}_p<\tilde{M}_p<\tilde{T}_p$,\\
\item $\frac{c_p}{a_p}\tilde{M}_p<\tilde{P}_p<\tilde{M}_p$,\\
\item $\frac{b_p}{a_p}\tilde{T}_p<\tilde{P}_p<\tilde{T}_p$,\\
\end{enumerate}
where $a_p=\pi/2$, and $b_p$ and $c_p$ are as defined in Theorem \ref{monoton-thm}.
\end{theorem}

\begin{proof} It is easy to see from the definition \eqref{pintegrals} and (\ref{pmeans}) that the 
following ratios of the means can be simplified as below:
$$f_1(z)=\frac{\int^z_0(1+t^p)^{-1}dt}{\int^z_0(1+t^p)^{-1/p}dt}=\frac{\tilde{M_p}}{\tilde{T_p}},\quad
f_2(z)=\frac{\int^z_0(1+t^p)^{-1/p}dt}{\int^z_0(1-t^p)^{-1/p}dt}=\frac{\tilde{P_p}}{\tilde{M_p}},$$
$$f_3(z)=\frac{\int^z_0(1+t^p)^{-1}dt}{\int^z_0(1-t^p)^{-1/p}dt}=\frac{\tilde{P_p}}{\tilde{T_p}}.$$
For the monotonicity of the functions $f_i,\,i=1,2,3$, we use the result given by
Cheeger et. al \cite[p.42]{cgt} that if $h_1,h_2:\mathbb{R}\to[0,\infty)$ are the integrable functions, and $h_1/h_2$
is decreasing then the function
$$x\mapsto \frac{\int_0^x{h_1(t)dt}}{\int_0^x{h_2(t)dt}}$$
is also decreasing. Clearly, the functions $f_i,\,i=1,2,3$ are decreasing, and the limiting values follows 
easily from the definitions. This completes the proof.
\end{proof}

\bigskip

\noindent{\bf Proof of Theorem \ref{monoton-thm}.}
Let $f_1(x)=f(x)/g(x)$, where $f(x)=\arcsin_p(x)$, $g(x)=\arcsin _q(x)$, $x \in (0,1)$, and $1<p<q$. 
Applying Lemma \ref{lem0}, for $a=0$, one has that $f_1(x)= (f(x)-f(0))/(g(x)-g(0))$. Now, after simple computations, we get
$$\frac{f'(x)}{g'(x)}= \frac{(1-x^q)^{1/q}}{(1-x^p)^{1/p}}= h_1(x).$$ Since  
$$h'_1(x)=\frac{(x^p-x^q)}{x(1-x^q)(1-x^p)}\cdot h_1(x) >0,$$ we get that $h_1(x)$ is strictly increasing  in $(0,1)$. This implies 
that $f_1(x)$ is strictly increasing, too. This fact implies the proof of part (1).

For the proof of (2), write $f_2(x)=f(x)/g(x)=(f(x)-f(0))/(g(x)-g(0))$, where $f(x)= {\rm arcsinh}_p(x)$ and $g(x)= {\rm arcsinh}_q(x)$. One has  
$$\frac{f'(x)}{g'(x)}= \frac{(1+x^q)^{1/q}}{(1+x^p)^{1/p}}= h_2(x).$$ After simple computations, we obtain   
$$(1+x^p)(1+x^q)\cdot h'_2(x)= (x^{q-1}-x^{p-1})h_2(x)<0,$$  as $q-1>0,\, p-1>0,\, q-1>p-1$ and $x \in (0,1)$. This means that $h'_2(x)<0$, 
so $h_2(x)$ is strictly decreasing; implying that $f_2(x)$ is strictly decreasing.

For (3), let $f_3(x)=f(x)/g(x)= {\rm arctanh}_p(x)/{\rm arctanh}_q(x)$. One has  
$f'(x)/g'(x)= (1-x^q)/(1-x^p) = h_3(x)$.  After simple computations, we see that
$$(1-x^p)^2\cdot h'_3(x)= x^{p-1}\cdot  u_1(x),$$  
where $u_1(x)= (q-p)x^q-qx^{q-p}+p$, here $u_1(0)= p>0$ and $u_1(1)= 0$.  On the other hand, $u'_1(x)= q(q-p)x^{q-1}( 1-x^{-p})$. 
Since $1-x^{-p}= (x^p-1)/x^p <0$  (by $p>0, x \in (0,1)$), we get that $u'_1(x)<0$. Thus $u_1(x)>u_1(1)=0$. Therefore, $h'_3(x)>0$, so $h_3(x)$ is strictly increasing. This implies that $f_3(x)$ is strictly increasing, too.

For the proof of part (4), let $f_4(x)= f(x)/g(x)=\arctan_p(x)/\arctan_q(x)$. One has $f'(x)/g'(x)= (1+x^q)/(1+x^p)= h_4(x)$. 
After simple computations, we conclude that
$$(1+x^p)^2 h'_4(x)=x^{p-1} u_2(x),$$ where  
$$u_2(x)=  qx^{q-p} +(q-p)x^q-p.$$ 
Here $u_2(0)=-p$, and $u_2(1)= 2(q-p)>0$, so $u_2(x)$ has at least a zero in $(0,1)$. We will show that, there is a single such zero. Indeed, one has  $$u'_2(x)= q(q-p)x^{q-p-1} + q(q-p)x^{q-1}>0,$$ 
so $u_2(x)$ is strictly increasing in $(0,1)$. Let $x_0$ be the single zero of $u_2(x)=0$. As $u_2(0)=-p$,  clearly $u_2(x)<0$ for 
$x \in (0,x_0)$ and similarly, $u_2(x)>0$ for $x\in (x_0,1)$.  As $h'_4(x) <0$, resp. $h'_4(x)>0$ in these intervals, the proof of (4) follows from the monotonicity of $h_4(x)$ and Lemma \ref{lem0}.
$\hfill\square$


\bigskip

\noindent{\bf Proof of Theorem \ref{PpMp1}.}
Let $f(t)= \sqrt{F(t)}$ and $g(t)=\sqrt{G(t)}$ in Cauchy-Bouniakowski inequality \eqref{(1)}, 
where $F(t),\, G(t)>0$.  Put $[a,b]= [0,x]$, Then one gets the inequality:
\begin{equation}\label{(1’)}
\left(\int_0^x{\sqrt{F(t)G(t)}}dt\right)^2 \leq \int_0^x{F(t)}dt \cdot \int_0^x{G(t)}dt.
\end{equation}
With the same notations, from the P\'olya-Szeg\"o inequality \eqref{(2)} one gets:
\begin{equation}\label{(2’)}
k(x,p)\left(\int_0^x{\sqrt{F(t)G(t)}}dt\right)^2 \geq \int_0^x{F(t)}dt \cdot \int_0^x{G(t)}dt,
\end{equation}
here $k(x,p)$ is as defined in Theorem \ref{PpMp1}. Let now $f(t)= (1-t^p)^{-1/p}$ and $g(t)= (1+t^p)^{-1/p}$. 
From \eqref{(1’)} and \eqref{(2’)} 
one obtains
\begin{equation}\label{(13)}
{\rm arcsin}_{2p}(x)^2 \leq   {\rm arcsin}_p(x)\,  {\rm arcsinh}_p(x),
\end{equation}
and
\begin{equation}\label{(14)}
{\rm arcsin}_p(x)\,{\rm arcsinh}_p(x)\leq  k(x,p){\rm arcsin}_{2p}(x)^2, 
\end{equation}
respectively.
 By definition, inequality \eqref{(13)} and \eqref{(14)} imply the proof of left hand-side and right-hand side 
of \eqref{(15-16)}, respectively.
$\hfill\square$

\bigskip

\noindent{\bf Proof of Theorem \ref{PpMp2}.}
Apply the Diaz-Metcalf inequality \eqref{(8)} for $f(t)=\sqrt{F(t)},\, g(t)=\sqrt{G(t)},\, [a,b]= [0,x]$, yielding
$$\int_0^x{G}dt+M \cdot m\cdot \int_0^x{F}dt \leq (M+m)\cdot \int_0^x{\sqrt{FG}}dt.$$    
Let  $F(t)= (1+t^p)^{-1/p},\,  G(t)= (1-t^p)^{-1/p}.$
Here $G(t)/F(t)=((1+t^p)/(1-t^p))^{1/p}$, which is strictly increasing. Thus  
$$m=1\leq  \sqrt{F/G} \leq ((1+x^p)/(1-x^p))^{1/(2p)}= M.$$ One obtains
\begin{equation}\label{(22)}
{\rm arcsin}_p(x) + M\cdot {\rm arcsinh}_p(x)\leq (M+1){\rm arcsin}_{2p}(x),
\end{equation}
this implies the proof of \eqref{(23)}.

Let $[a,b]=[0,x]$ and $f(t)= (1+t^p)^{-1}$ and $g(t)= (1-t^p)^{-1}$. 
As $f(t)+g(t)=  2/(1-t^{2p})$, applying the  Minkowski inequality \eqref{(7)} for $t=1/p,\,p>1$, we get
\begin{equation}\label{(24)}
{\rm arcsin}_p(x)^p + {\rm arcsinh}_p(x)^p \leq 2 \left(\int_0^x{A^2}dt\right)^p,
\end{equation}
where $A(t)= 1/(1-t^{2p})^{1/(2p)}$. 
Clearly,  $\int_0^x{A(t)}dt={\rm arcsin}_{2p}(x)$, so for obtaining an upper bound for $\int_0^x{A(t)^2}dt$, we apply the 
P\'olya-Szeg\H{o} inequality for $f(t)= 1/(1-t^{2p})^{1/p}$ and $g(t)= 1$.  Since in this case one has  $1\leq f(t)\leq 1/(1-x^{2p})^{1/p}$, we get from \eqref{(2)}
$$ \int_0^x{A(t)^2}dt \leq {\rm arcsin}_{2p}(x)^2R(x,p),$$
By using \eqref{(24)}, finally we get
\begin{equation}\label{(25)}
\frac{x^p({\rm arcsin}_p(x)^p + {\rm arcsinh}_p(x)^p)}{{\rm arcsin}_{2p}(x)^{2p}} \leq R(x,p),
\end{equation}
this implies inequality \eqref{(26)}.
$\hfill\square$

\bigskip

\noindent{\bf Proof of Theorem \ref{TpLp1}.} Apply the Chebyshev inequality \eqref{(4)} for the 
functions (1) and (2) of Lemma \ref{Lemma1j}, 
which are of different type of monotonicity. One obtains the following inequality
\begin{equation}\label{(19)}
x\cdot {\rm arctanh}_{2p}(x)   \geq   {\rm arctan}_p(x) {\rm arctanh}_p(x). 
\end{equation}
This implies \eqref{(20)}. The proof of \eqref{(20b)} follows if we apply the Gr\"uss 
inequality for the same functions as above and utilize relation \eqref{(19)}.
$\hfill\square$

\bigskip

\noindent{\bf Proof of Theorem \ref{TpLp-single1}.}
Applying Cauchy-Bouniakowski inequality \eqref{(1)} for $f(t)=\sqrt{F(t)}$, and  $g(t)= 1/\sqrt{F(t)}$, 
we get the following inequality
\begin{equation}\label{(27)}
\int_a^b {F(t)}dt \cdot \int_a^b {1/F(t)}dt \geq (b-a)^2.
\end{equation}
Applying the same notations as above for the P\'olya-Szeg\H{o} inequality \eqref{(2)}, 
one obtains the inequality (called also as Schweizer  inequality).
Suppose that   $0<\alpha< F(t) < A$. Then
\begin{equation}\label{(28)}
\int_a^b {F(t)}dt \cdot \int_a^b {1/F(t)}dt \geq  \frac{(b-a)^2 (\alpha+A)^2}{4\alpha A}.    
\end{equation}
Let now $F(t)= 1+t^p$, with $t \in [a,b]= [0, x]$  in \eqref{(27)} and \eqref{(28)}. As $\alpha=1$, 
$A= 1+x^p$, one obtains the following double inequality
\begin{equation}\label{(29)}
\frac{4(1+x^p)(1+ x^p/(p+1)}{(x^p+2)^2}\leq \frac{x}{{\rm arctan}_p(x)}  \leq 1+\frac{x^p}{1+p}  
\end{equation}
and
\begin{equation}\label{(30)}
\frac{4(1-x^p)(1- x^p/(1+p)}{(2-x^p)^2}\leq \frac{x}{{\rm arctanh}_p(x)} \leq 1-\frac{x^p}{1+p}.     
\end{equation}
The right side of \eqref{(29)} and \eqref{(30)} are obtained from \eqref{(27)}, while the left side of 
\eqref{(29)} and \eqref{(30)} are obtained from \eqref{(28)}.
This completes the proof of Theorem \ref{TpLp-single1}.
$\hfill\square$

\bigskip

\noindent{\bf Proof of Theorem \ref{TpLp-single2}.}
Letting  $t^p=u$ one has $dt= (1/p)u^{1/p-1} du$, and applying inequality \eqref{(27)} and \eqref{(28)} for $F(t)= (1-t^p)^{1/p}$ we get 
$$\int_0^x {F(t)}dt= (1/p) \int_0^{x^p}{u^{1/p-1}(1-u)^{1/p}}du.$$
  As   
$$\int_a^b{u^{a-1}(1-u)^{b-1}}du= B(a,b) B(a,b:x), $$ where $B(a,b)$ is the beta function, and $B(a,b:x)$ 
is the incomplete beta function, we get

\begin{equation}\label{23092018}
\int_0^x{ (1-t^p)^{1/p}}dt =  (1/p)\cdot B(1/p, 1+ 1/p) B(1/p, 1+1/p: x^p)
\end{equation}
Applying \eqref{23092018}, and utilizing \eqref{(27)} and \eqref{(28)}, we get the following double inequality 
$$\frac{p\,x^2}{B\left(1/p, 1+1/p\right)} \leq {\rm arcsin}_p(x) \leq  \frac{p\,x^2 \,B(1/p, 1+1/p)(2-x^p)^2]}{4(1-x^p)B(1/p, 1+1/p)]}.$$
This implies \eqref{APp}. For the proof of \eqref{AMp},
we apply \eqref{(27)} and \eqref{(28)} for 
$F(t)= (1+t^p)^{1/p}$ and get the following double inequality
$$\frac{x^2}{j(x,p)}\leq  {\rm arcsinh}_p(x) \leq \frac{x^2(1+(1+x^p)^{1/p})}{4j(x,p)(1+x^p)^{1/p}}.$$  
This completes the proof of Theorem \ref{TpLp-single2}.
$\hfill\square$

\bigskip

%

We finish this paper by giving the following remark. 
\begin{remark} \rm In \cite{nueman1}, Neuman studied the $p$-version 
of Schwab-Borchardt mean $S_p,\,p>1$, which was expressed in terms of hypergeometric function as follows,
\begin{equation}\label{hyp-p} 
\frac{y}{S_p(x,y)}=F\left(\frac{1}{p},\frac{1}{p},1+\frac{1}{p},1-\left(\frac{x}{y}\right)^p\right),\end{equation}
(\cite[(9)]{nueman2}). Here we give a proof, which leads us to formula \cite[(22)]{nueman1}. 
For $x>y>0,\,p>1$, (\ref{hyp-p}) can be written as 
$$
\frac{y}{S_p(x,y)}=
F\left(\frac{1}{p},\frac{1}{p},1+\frac{1}{p},-w^p\right)
$$
where $w=\left(\frac{x^p-y^p}{y^p}\right)^{1/p}$.
By applying the following transformation formula (see \cite[15.3.5]{as})
$$F(a,b;c;z)=(1-z)^{-b}F\left(b,c-a;c;-\frac{z}{1-z}\right),$$ \cite[Lemma 1]{bbp} and the identity 
${\rm arcsinh}_p(\sqrt[p]{t^p-1})={\rm arccosh}_p(t)$
we get
\begin{eqnarray*}
\frac{y}{S_p(x,y)}&=& \frac{w}{w(1+w^p)^{1/p}}F\left(\frac{1}{p},\frac{1}{p},1+\frac{1}{p},\frac{w^p}{1+w^p}\right)\\
&=&\frac{{\rm arcsinh}_p(w)}{w}=\frac{{\rm arcsinh}_p\left(\frac{x^p-y^p}{y^p}\right)^{1/p}}
{\left(\frac{x^p-y^p}{y^p}\right)^{1/p}}\\
&=&\frac{y\,{\rm arcsinh}_p\sqrt[p]{(x/y)^p-1}}{(x^p-y^p)^{1/p}}=
\frac{y\,{\rm arccosh}_p(x/y)}{\sqrt[p]{(x^p-y^p}}.
\end{eqnarray*}
The case when $0<x<y$ follows similarly.
\end{remark}



\end{document}